\documentclass[12pt]{amsart}
\usepackage{amsthm,amsmath,amssymb,url,verbatim}
\usepackage[margin=1.25in]{geometry}
{\providecommand{\noopsort}[1]{} 

\theoremstyle{plain}
\newtheorem{theorem}{Theorem}[section]
\newtheorem{corollary}[theorem]{Corollary}
\newtheorem{lemma}[theorem]{Lemma}

\newtheorem{main}{Theorem}
\newtheorem{maincor}[main]{Corollary}
\theoremstyle{definition}

\theoremstyle{remark}

\numberwithin{equation}{section}

\newcommand{\ep}{\epsilon}

\newcommand{\Z}{\mathbb{Z}}

\newcommand{\s}{\mathbb{S}}

\DeclareMathOperator{\SU}{SU}

\DeclareMathOperator{\cod}{cod}
\DeclareMathOperator{\diag}{diag}

\newcommand{\of}[1]{\left(#1\right)}
\newcommand{\ceil}[1]{\left\lceil #1 \right\rceil}
\newcommand{\floor}[1]{\left\lfloor #1 \right\rfloor}

\newcommand{\id}{\mathrm{id}}

\title{On a generalized conjecture of Hopf with symmetry}
\author{Manuel Amann}
\author{Lee Kennard}
\date{\today}
\thanks{The first author was supported by a grant of the German Research Foundation. The second author was partially supported by National Science Foundation Grant DMS-1045292.}

\begin{document}

\begin{abstract}
A famous conjecture of Hopf is that the product of the two-dimensional sphere with itself does not admit a Riemannian metric with positive sectional curvature. More generally, one may conjecture that this holds for any nontrivial product. We provide evidence for this generalized conjecture in the presence of symmetry.
\end{abstract}

\maketitle

\bigskip

Among compact, simply connected, even-dimensional smooth manifolds, the only examples known to admit a Riemannian metric with positive sectional curvature form a short list: spheres, complex projective spaces, quaternionic projective spaces, the Cayley plane, the three flag manifolds discovered by Wallach \cite{Wallach72}, and the biquotient $\SU(2)/\!/T^2$ discovered by Eschenburg \cite{Eschenburg84}.

In order to find additional examples, it is natural to look among metrics with symmetry. This strategy has recently resulted in a new example in dimension seven (see Dearricott \cite{Dearricott11} and Grove--Verdiani--Ziller \cite{GroveVerdianiZiller11}). To narrow the search, one seeks topological obstructions to positive curvature and symmetry. This broad research program was formulated by Grove and developed by him and many others over the past two decades (see Grove \cite{Grove09}, Wilking \cite{Wilking07}, and Ziller \cite{Ziller07,Ziller??} for surveys).

In this article, we prove further topological restrictions in the presence of torus symmetry. Our first theorem considers the case where the positively curved Riemannian manifold $M^{2n}$ has vanishing fourth Betti number. To motivate this assumption, we recall that, if the rank of the isometric torus action exceeds $\log_{4/3}(2n)$, then the Betti numbers of $M$ satisfy $b_2(M)\leq b_4(M) \leq 1$ (see \cite{Kennard2}).

\begin{main}\label{thm:Euler}
Let $M^{2n}$ be a simply connected, closed manifold with $b_4(M) = 0$. Assume $M$ admits a Riemannian manifold with positive sectional curvature invariant under the action of a torus $T$ with $\dim(T) \geq \log_{4/3}(2n)$. The following hold:
	\begin{enumerate}
	\item The Euler characteristic satisfies $\chi(M) = \chi(\s^{2n}) = 2$.
	\item The signature satisfies $\sigma(M) = \sigma(\s^{2n}) = 0$.
	\item The fixed-point set $M^T$ is an even-dimensional rational sphere.
	\item For $g \in T$, $M^g$ is nonempty, and the number of components is at most two, with equality only if $g$ is an involution.
	\end{enumerate}
\end{main}

As an application of this result, consider an arbitrary simply connected, closed manifold $N^n$, and consider its two-fold product $M^{2n} = N\times N$. Suppose that $M$ has a metric with positive curvature and an isometric torus ation of rank at least $\log_{4/3}(2n)$. As mentioned above, this implies $b_2(M) \leq b_4(M) \leq 1$. By the K\"unneth formulas, $b_4(M) = 0$. By Theorem \ref{thm:Euler},
	\[2 = \chi(M) = \chi(N)^2,\]
which is impossible. Hence $N \times N$ has no such metric. A similar conclusion can be drawn for connected sums. We summarize this corollary as follows:

\begin{maincor}
Let $N^n$ be a simply connected, closed manifold. The product $N\times N$ does not admit a Riemannian metric with positive sectional curvature and an isometric torus action of rank at least $\log_{4/3}(2n)$.

Similarly, if $n$ is even and $\chi(N) \neq 2$, the connected sum $N \# N$ does not admit a positively curved metric invariant under a torus action of rank at least $\log_{4/3}(n)$.
\end{maincor}

Hopf conjectured that $\s^2 \times \s^2$ does not admit a Riemannian metric with positive sectional curvature. This corollary can be seen as positive evidence for the generalized conjecture that no product $N\times N$ admits such a metric. One can similarly prove that $M$ cannot be the total space of a fibration with base $N$ and fiber $N$, at least when $b_3(N) = b_5(N) = 0$.

Another way to generalize Hopf's conjecture is to consider $\s^2 \times \s^2$ as a rank-two symmetric space. The simply connected, compact, rank-one symmetric spaces are the spheres and projective spaces, and these admit positive sectional curvature. One might conjecture that, among compact, simply connected symmetric spaces, only those with rank one admit positive sectional curvature. The next corollary can be seen as evidence for this.

\begin{maincor}
Suppose $M^{2n}$ has the rational cohomology of a compact, simply connected symmetric space $N$. Suppose that $M$ admits a Riemannian metric with positive sectional curvature which is invariant under an effective $r$--torus action. If $b_4(M) = 0$ and $r \geq \log_{4/3}(2n)$, then $N$ is a sphere.
\end{maincor}

The proof of this corollary is an immediate consequence of Theorem \ref{thm:Euler} and \cite[Theorem A]{Kennard2}.

\smallskip

The conclusion of Theorem \ref{thm:Euler} can be improved by imposing additional topological conditions on $M$. For example, suppose that $M$ is rationally elliptic, as conjectured by Bott (see Grove--Halperin \cite{GroveHalperin82}). Since $\chi(M) = 2$, the odd Betti numbers vanish (see \cite{GroveHalperin82}), hence $M$ is a rational sphere. We summarize similar corollaries here. For the proof, see Section \ref{sec:HomeoDiffeo}.

\begin{maincor}\label{cor:HomeoDiffeo}
Let $M^{2n}$ be a simply connected, closed Riemannian manifold with $b_4(M) = 0$. Assume $M$ admits a metric with positive sectional curvature and an isometric torus action of rank at least $\log_{4/3}(2n)$. The following hold:
	\begin{enumerate}
	\item If $M$ has vanishing odd-dimensional rational cohomology, e.g., if $M$ is rationally elliptic, then $M$ is a rational $\s^{2n}$.
	\item If $M$ is $p$--elliptic for some prime $p \geq 2n$, then $M$ is a mod $p$ homology $\s^{2n}$.
	\item If $M$ has vanishing homology in odd degrees, then $M$ is homeomorphic to $\s^{2n}$.
	\item If $M$ is a biquotient, then $M$ is diffeomorphic to $\s^{2n}$.
	\item If $M$ admits a smooth, effective cohomogeneity one action by a compact, connected Lie group, and if the homology of $M$ has no two--torsion, then $M$ is equivariantly diffeomorphic to $\s^{2n}$ equipped with a linear $G$--action.
	\end{enumerate}
\end{maincor}

We remark that the torus action in this corollary need not respect the biquotient or cohomogeneity one structure. 

We also remark that, whenever $M$ is spin and homeomorphic to $\s^{4k}$, its elliptic genus vanishes. Corollary \ref{cor:HomeoDiffeo} can therefore be seen as providing further evidence for a conjecture Dessai (see Dessai \cite{Dessai05,Dessai07}, Weisskopf \cite{Weisskopf}, and \cite{AmannKennard1}).

\smallskip

We conclude by remarking on the assumption that $b_4(M) = 0$. In the presence of positive curvature and torus symmetry as in Theorem \ref{thm:Euler}, the only other possibility is $b_4(M) = 1$. Since our results when $b_4(M) = 0$ suggest that $M$ might be a rational sphere, one might similarly hope to show that $b_4(M) = 1$ implies that $M$ has the rational type of a projective space. In particular, one might hope to calculate the Euler characteristic and signature of such a manifold.

\subsection*{Acknowledgements}
We would like to thank Burkhard Wilking for a motivating discussion on calculating Euler characteristics of positively curved manifolds with symmetry. We are also grateful to Wolfgang Ziller for alerting us to Asoh's paper.

\bigskip\section{Preliminaries}\bigskip

In this section, we prove the key lemma (Corollary \ref{cor:CoverTwo}) for the proof of Theorem \ref{thm:Euler}. There are two technical lemmas required. The first is essentially Lemma 5.6 in \cite{Kennard3}:

\begin{lemma}
Let $M^n$ be a closed, simply connected, positively curved Riemannian manifold, let $T$ be a torus acting effectively on $M$, and let $x$ be a fixed point. Fix $c \geq 1$ and $k_0 \leq \frac{n-c}{4}$. Set $j = \floor{\log_2(k_0)}+1$ or $j = \floor{\log_2(k_0)}$ according to whether $n$ is even or odd. If there exist independent involutions $\iota_1,\ldots,\iota_j \in T$ such that $M^{\iota_i}_x$ has codimension at most $\frac{n-c}{2}$ for all $i$, then one of the following holds:
	\begin{itemize}
	\item $M$ has $4$--periodic rational cohomology, or
	\item there exists an involution $\iota \in T^s$ such that $\of{M^{\iota_1}_x}^T \cap \cdots \cap \of{M^{\iota_j}_x}^T  \subseteq M^\iota_x$, $k_0 < \cod\of{M^\iota_x} \leq \frac{n-c}{2}$, and $\dim\ker\of{T|_{M^\iota_x}} \leq 1$.
	\end{itemize}
\end{lemma}

Here, $\dim\ker\of{T|_{M^\iota_x}}$ denotes the dimension of the kernel of the induced $T$--action on $M^\iota_x$. For the definition of periodic cohomology in this context, see \cite[Definition 1.8]{AmannKennard1}.

The only difference between this lemma and Lemma 5.6 in \cite{Kennard3} is that we include here the condition that $M^\iota_x$ contains every $\of{M^{\iota_i}_x}^T$. This requires one additional argument in the proof. Namely, if $M^\sigma_x$ and $M^\tau_x$ each contain $\of{M^{\iota_i}_x}^T$ and have codimension less than $\frac{n}{2}$, then $M^\sigma_x \cap M^\tau_x$ is connected by Wilking's connectedness lemma (see \cite[Theorem 2.1]{Wilking03}). In particular, for all $i$,
	\[\of{M^{\iota_i}_x}^T \subseteq M^\sigma_x \cap M^\tau_x = M^{\langle\sigma,\tau\rangle}_x \subseteq M^{\sigma\tau}_x.\]
The original proof now shows that, if neither of the two conclusions occur, then every $\iota$ in the subgroup of $T$ generated by the $\iota_i$ satsifies $\cod\of{M^\iota_x} \leq k_0$, $\dim\ker\of{T|_{M^\iota_x}} \leq 1$, and $\of{M^{\iota_1}_x}^T \cap \cdots \cap \of{M^{\iota_j}_x}^T \subseteq M^\iota_x$. As in the original proof, this situtation actually implies that $M$ is rationally $4$--periodic, which completes the proof.

To apply this lemma, one must prove the existence of the involutions $\iota_1,\ldots,\iota_j$. To do this, we simultaneously generalize Proposition 1.14 in \cite{AmannKennard1} and Lemma 5.7 in \cite{Kennard3}. The result is the following.

\begin{lemma}
Let $n \geq c \geq 0$, $j \geq 1$, and $t \geq 1$. Let $M^n$ be a closed, positively curved Riemannian manifold, assume $T^s$ acts effectively by isometries on $M$, and let $x_1,\ldots,x_t \in M$ be fixed points. If
	\begin{equation}\label{eqn:Griesmer}
	t\floor{\frac{n}{2}} < j - 1 + \sum_{i=0}^{s-j} \ceil{2^{-i}\ceil{\frac{t(n-c)+1}{4}}},
	\end{equation}
then there exist independent involutions $\iota_1,\ldots,\iota_j\in T^s$ such that, for all $1\leq i \leq j$, the maximal component of $M^{\iota_i}$ has codimension at most $\frac{n-c}{2}$ and contains at least $\ceil{\frac{t+1}{2}}$ of the points $x_1,\ldots,x_t$.
\end{lemma}

\begin{proof}
Set $m = \floor{\frac{n}{2}}$. For each $x_i$, choose a basis of $T_{x_i} M$ such that the image of every $\iota \in \Z_2^s \subseteq T^s$ under the isotropy representation takes the form $\diag(\ep_1 I, \ldots, \ep_m I)$ or $\diag(\ep_1 I, \ldots, \ep_m I, 1)$ according to whether $n$ is even or odd. Here, the $\ep_i = \pm 1$, and $I$ denotes the two-by-two identity matrix. Observe that $\cod\of{M^\iota_{x_i}}$ equals twice the Hamming weight of $(\ep_1,\ldots,\ep_m) \in \Z_2^m$.

The direct sum of these $t$ maps induces a homomorphism $\phi:\Z_2^s \to \bigoplus_{i=1}^t \Z_2^m \cong \Z_2^{tm}$. Let $\phi_u$ denote the composition of $\phi$ with the projection onto the $u$--th component. For example, the codimension of $M^\iota_{x_1}$ is equal to twice the Hamming weight of the vector $(\phi_1(\iota), \ldots, \phi_m(\iota)) \in \Z_2^m$.

Consider now an integer $0 \leq h \leq j - 1$ such that there exist independent $\iota_1,\ldots,\iota_h\in \Z_2^s$ and integers $u_1,\ldots,u_h$ such that, for all $1 \leq i \leq h$,
	\begin{enumerate}
	\item there is a component $N_i$ of $M^{\iota_i}$ with codimension at most $\frac{n-c}{2}$ that contains at least $\ceil{\frac{t+1}{2}}$ of the points $x_1,\ldots,x_t$,
	\item $\phi_{u_i}(\iota_i) \in \Z_2$ is nontrivial, and 
	\item $\phi_{u_{i'}}(\iota_i) \in \Z_2$ is trivial for all $1 \leq i' < i$.
	\end{enumerate}
Note that these conditions are vacuously satisfied for $h = 0$. We claim that, given $\iota_1,\ldots,\iota_h$ as above, there exists $\iota_{h+1}$ such that all of these properties hold. By induction, this suffices to prove the existence of $\iota_1,\ldots,\iota_j$ as in the conclusion of the lemma.

To start, choose a $\Z_2^{s-h} \subseteq \ker(\phi_{u_1}) \cap \cdots \cap \ker(\phi_{u_h}) \subseteq \Z_2^s$. Note that every $\iota \in \Z_2^{s-h}$ automatically satisfies the last condition above. Moreover, every nontrivial $\iota \in \Z_2^{s-h}$ is independent of $\iota_1,\ldots,\iota_h$ and is nontrivial since the $T^s$ action is effective. It therefore suffices to prove that some $\iota \in \Z_2^{s-h}$ has a fixed-point component with codimension at most $\frac{n-c}{2}$ that contains $\ceil{\frac{t+1}{2}}$ of the $x_1,\ldots,x_t$.

Consider the composition 
	\[\Z_2^{s-h} \subseteq \Z_2^s \stackrel{\phi}{\longrightarrow} \Z_2^{tm} \longrightarrow\Z_2^{tm - h},\]
where the last map projects away the $u_i$--th components for $1 \leq i \leq h$. By the choice of the $u_i$, the Hamming weight of the image of $\iota \in \Z_2^{s-h}$ under this composition is half of the sum of the codimensions $k_i = \cod\of{M^\iota_{x_i}}$. As in the proof of Proposition 1.14 in \cite{AmannKennard1}, an argument based on Frankel's theorem implies the following: If $\iota \in \Z_2^{s-h}$ exists such that $\sum_{i=1}^t k_i \leq \frac{t(n-c)}{2}$, then there exists a component of $M^\iota$ with codimension at most $\frac{n-c}{2}$ that contains $\ceil{\frac{t+1}{2}}$ of the $x_1,\ldots,x_t$. It therefore suffices to prove that some nontrivial $\iota \in \Z_2^{s-h}$ exists whose image under the above map $\Z_2^{s-h} \to \Z_2^{tm-h}$ has weight at most $\frac{t(n-c)}{4}$.

If no such involution exists, the Griesmer bound (see, for example, the proof of Proposition 1.14 in \cite{AmannKennard1}) implies that
	\[tm - h \geq \sum_{i=0}^{s - h - 1} \ceil{2^{-i} \ceil{\frac{tn-tc+1}{4}}}.\]
Since every summand on the right-hand side is at least one, this inequality is preserved if we replace $h$ by $h+1$. Inductively, this inequality is preserved if we replace $h$ by $j-1$. On the other hand, this contradicts Inequality \ref{eqn:Griesmer}, so the proof is complete.
\end{proof}

The last preliminary result is a consequence of these two lemmas. It is the only result of this section we will use in the proof of Theorem \ref{thm:Euler}.

\begin{corollary}\label{cor:CoverTwo}
Let $M^n$ be a closed, positively curved Riemannian manifold with $n \geq 21$. Assume $T^s$ acts effectively by isometries on $M$ with $s \geq \log_{4/3}(n-3)$. If $x$ and $y$ are fixed by $T^s$, and if $M$ is not rationally $4$--periodic, then there exists an involution $\iota \in T^s$ and a component $N \subseteq M^\iota$ such that $\frac{n-4}{4} < \cod(N) \leq \frac{n-4}{2}$, $\dim\ker\of{T|_N} \leq 1$, and $x,y\in N$. 
\end{corollary}

\begin{proof}
Set $c = 4$, $k_0 = \frac{n-4}{4}$, and $j = \floor{\log_2(k_0)} + 1 - \ep = \floor{\log_2(n-4)} - 1 - \ep$, where $\ep$ is zero or one according to whether $n$ is even or odd. By the first lemma above, it suffices to prove the existence of independent involutions $\iota_1, \ldots, \iota_j \in T^s$ such that each $M^{\iota_i}$ has a component of codimension at most $k_0$ that contains both $x$ and $y$.

Set $t = 2$, and note that $\ceil{\frac{t+1}{2}} = 2$. By the second lemma, such a collection of involutions exists if Inequality \ref{eqn:Griesmer} holds.

To verify this inequality for all $n \geq 21$, first observe that \[s-j + 1\geq \ceil{\log_{4/3}(n-3)} - j + 1 \geq \log_2(n-3)\] and hence that $\ceil{\frac{n-3}{2^{i+1}}} = 1$ for all $i \geq s - j$. In particular, we can estimate the right-hand side, denoted by $R$, of Inequality \ref{eqn:Griesmer} as follows. First,
	\[R = j - 1 + \sum_{i=0}^{s-j}\ceil{2^{-i}\ceil{\frac{2(n-4)+2}{4}}}
	   \geq j-1 + \sum_{i=0}^{s-j}\ceil{\frac{n-3}{2^{i+1}}}
	   = \sum_{i=0}^{s-1} \ceil{\frac{n-3}{2^{i+1}}}.\]
Second, note that $s \geq 6$, hence we can write
	\[R \geq 5 + \sum_{i=0}^{s-6} \frac{n-3}{2^{i+1}}.\]
Finally, evaluating the geometric sum implies that $R > n + 1$ since $s - 5 > \log_2(n-3)$. This proves that Inequality \ref{eqn:Griesmer} holds.
\end{proof}

\bigskip\section{Proof of Theorem A}\label{sec:Euler}\bigskip

For an isometric action by a torus $T$ on a Riemannian manifold $M$, the fixed-point set $M^T$ is a union of closed, totally geodesic submanifolds of even codimension. We recall the following results that relate the topology of $M$ and $M^T$ (see, for example, \cite{Conner57} and \cite[p. 72]{Hirzebruch92}):
	\begin{itemize}
	\item (Conner's theorem) The even Betti numbers satisfy $\sum b_{2i}(M^T) \leq \sum b_{2i}(M)$ and likewise for the odd Betti numbers.
	\item The Euler characteristics satisfy $\chi(M) = \chi(M^T)$
	\item The signatures satisfy $\sigma(M) = \sigma(M^T)$.
	\end{itemize}

Note that, if $M$ is an even-dimensional, positively curved rational sphere, then $M^T$ is as well. Since a positive- and even-dimensional sphere trivially has Euler characteristic two and signature zero, the first three conclusions of Theorem \ref{thm:Euler} are an immediate consequence of the following:

\begin{theorem}\label{thm:EulerPLUS}
Let $M^n$ be a closed, simply connected, positively curved Riemannian manifold with $b_4(M) = 0$. If a torus $T$ acts isometrically and effectively on $M$ with $\dim(T) \geq \log_{4/3}(n-3)$, then $M^T = N^T$ for some totally geodesic, even-codimensional, positive-dimensional rational sphere $N \subseteq M$ to which the $T$--action restricts.
\end{theorem}

We deduce the final conclusion of Theorem \ref{thm:Euler} at the end of the section.

This result can be seen as a less localized version of the main theorem in \cite{Kennard2}. Under the assumptions of Theorem \ref{thm:EulerPLUS}, the results in \cite{Kennard2} imply that each component of $M^T$ is a component of the fixed-point set of some $N$ as in this theorem. The novelty here is that the entire fixed-point set $M^T$ is contained in this submanifold $N$.

The proof of Theorem \ref{thm:EulerPLUS} is by induction over the dimension in the style of Wilking \cite{Wilking03}. First, if $n \leq 3$, the result is trivial since $M$ is a homotopy sphere. Second, if $n = 4$, the result is vacuous since $b_4(M) = 1$ in this case. Third, for $5 \leq n \leq 20$, $s \geq \frac{n}{2}$, hence the result of Grove and Searle implies that $M$ is diffeomorphic to the sphere (see \cite{GroveSearle94}). In particular, $M$ is a rational sphere, and the theorem holds by taking $N = M$.

Finally, suppose that $n \geq 21$. As in the previous paragraph, if $M$ itself is a rational sphere, then the theorem immediately holds. We assume throughout that $M$ is not a rational sphere. The induction step has two parts. The first considers the case where some component $M^T_x$ has positive dimension.

\begin{lemma}
If some component $M^T_x$ of $M^T$ has positive dimension, then $M^T = M^T_x$ and is a rational sphere. In particular, the theorem holds with $N = M^T$.
\end{lemma}

\begin{proof}
Let $y \in M^T$. By Corollary \ref{cor:CoverTwo}, there exists an $\iota \in T$ and a component $N\subseteq M^\iota$ such that $\frac{n-4}{4} < \cod(N) \leq \frac{n-4}{2}$, $\dim\ker\of{T|_N} \leq 1$, and $x,y\in N$. By Wilking's connectedness lemma, $N$ is simply connected and $b_4(N) = 0$. Moreover, since $\cod(N) \geq \frac{n-3}{4}$, 
	\[\dim\of{T/\ker\of{T|_N}} \geq \dim(T) - 1 \geq \log_{4/3}(n-3) - 1 \geq \log_{4/3}\of{\dim N - 3}.\]
Since $T/\ker\of{T|_N}$ is a torus that acts effectively on $N$, a closed, simply connected, positively curved Riemannian manifold with $b_4(N) = 0$, the induction hypothesis applies to $N$. Since $M^T_x$ and $M^T_y$ are components of $N^T = N\cap M^T$, we therefore have
	\[\sum b_i(M^T_x \cup M^T_y) \leq \sum b_i(N^T) = 2.\]
Observe that $\sum b_i(M^T_x) \geq 2$ since $M^T_x$ has positive dimension. In particular, if $y$ lies in a component of $M^T$ different from $M^T_x$, then the the left-hand side is at least
	\[\sum b_i(M^T_x) + \sum b_i(M^T_y) \geq 2 + 1,\]
a contradiction. This shows that $M^T = M^T_x$. Moreover, the above argument works with $x = y$, hence $\sum b_i(M^T_x) = 2$, which implies that $M^T_x$ is a rational sphere.
\end{proof}

The other possibility is that $T$ has only isolated fixed points.

\begin{lemma}
If $\dim(M^T) = 0$, then there are exactly two isolated fixed points. Moreover, there exists a totally geodesic, positive-dimensional, and even-dimensional rational sphere $N \subseteq M$ on which $T$ acts such that $M^T = N^T$.
\end{lemma}

\begin{proof}
First, a theorem of Berger implies there is at least one fixed point (see \cite{Berger66}), and there cannot be exactly one fixed point (see Bredon \cite[Corollary IV.2.3, p. 178]{Bredon72}).

Suppose for a moment that $M^T$ has exactly two isolated fixed points. By Corollary \ref{cor:CoverTwo}, $M^T \subseteq P$ for some totally geodesic, even-dimensional, closed submanifold $P$ with $b_4(P) = 0$ and an isometric torus action of rank at least $\log_{4/3}\of{\dim P - 3}$. By the induction hypothesis applied to $P$, $P^T = N^T$ for some totally geodesic, even-dimensional, positive-dimensional rational sphere $N \subseteq P$. Since $M^T \subseteq N^T \subseteq M^T$, this proves the lemma in this case.

Finally, suppose that $M^T$ has at least three (distinct) isolated fixed points, $x$, $y$, and $z$. Two applications of Corollary \ref{cor:CoverTwo} imply the existence of involutions $\sigma,\tau\in T^s$ such that
	\begin{itemize}
	\item $M^\sigma_x$ has $\dim\ker\of{T|_{M^\sigma_x}} \leq 1$, $\frac{n-4}{4} < \cod\of{M^\sigma_x} \leq \frac{n-4}{2}$, and $y \in M^\sigma_x$, and
	\item $M^\tau_x$ has $\dim\ker\of{T|_{M^\tau_x}} \leq 1$, $\frac{n-4}{4} < \cod\of{M^\tau_x} \leq \frac{n-4}{2}$, and $z \in M^\sigma_x$.
	\end{itemize}
As before, the induction hypothesis applies to both $M^\sigma_x$ and $M^\tau_x$. Since the torus action on $M$ has only isolated fixed points, the same is true of the torus action on $M^\sigma_x$ and $M^\tau_x$. Hence $\of{M^\sigma_x}^T = \{x,y\}$ and $\of{M^\tau_x}^T = \{x,z\}$. This further implies
	\[\of{M^\sigma_x \cap M^\tau_x}^T 
	= \of{M^\sigma_x}^T \cap \of{M^\tau_x}^T = \{x\}.\]
On the other hand, Frankel's theorem implies that $M^\sigma_x$ and $M^\tau_x$ intersect, and Wilking's connectedness lemma implies that $M^\sigma_x \cap M^\tau_x$ is connected and simply connected. In particular, as with $M$, this intersection cannot have exactly one fixed point. This concludes the proof.
\end{proof}

This concludes the proof of Theorem \ref{thm:EulerPLUS} and hence of the first three conclusions of Theorem \ref{thm:Euler}. We now prove the fourth conclusion.

\begin{proof}[Proof of Theorem \ref{thm:Euler}, Conclusion (d)]
Let $g \in T$. The Lefschetz number of the map $g:M \to M$ is equal to $\chi(M)$ since $g$ is homotopic to the identity. Hence $\chi(M^g) = \chi(M) = 2$, so $M^g$ is nonempty.

Since $g:M \to M$ is orientation-preserving, each component of $M^g$ is a closed, even-dimensional, totally geodesic submanifold of $M$ to which the $T$--action restricts. Hence each component of $M^g$ contains an fixed point of $T$ by Berger's theorem. Since $M^T$ is a rational sphere, either $M^g$ is connected or it has two components and the $T$--actions on them have exactly one fixed point each. The latter case can only occur if each component of $M^g$ is non-orientable (see again Bredon \cite[Corollary IV.2.3, p. 178]{Bredon72}), which in turn can only occur if $g^2 = \id$.
\end{proof}

\smallskip\section{Proof of Corollary D}\label{sec:HomeoDiffeo}\bigskip

Corollary \ref{cor:HomeoDiffeo} is an easy consequence of Theorem \ref{thm:Euler} together with a few general classification results. We only use Theorem \ref{thm:Euler} to deduce that $\chi(M) = 2$.

First, it is immediate that, if $M$ has vanishing odd-dimensional rational homology, then $M$ is a rational homology sphere. Since $M$ is simply connected, this is equivalent to $M$ having the rational homotopy type of a sphere. 

Second, we refer to Powell \cite{Powell97} for a definition of $p$--elliptic for a prime $p$. In particular, it follows that $M$ is rationally elliptic and hence a rational sphere. In \cite[Theorem 1]{Powell97}, Powell classified $p$--elliptic rational spheres for $p \geq \dim(M)$, and it follows immediately that $M$ is a mod $p$ homology sphere.

Third, if $M$ has vanishing odd-dimensional integral homology, then its homology is torsion-free. Since $M$ is a rational homology sphere, $M$ is, in fact, an integral homology sphere. Since $M$ is simply connected, it follows from the resolution of the Poincar\'e conjecture that $M$ is homeomorphic to a sphere.

Fourth, since biquotients are rationally elliptic, $M$ is again a rational sphere. Totaro \cite[Theorem 6.1]{Totaro02} and Kapovitch--Ziller \cite[Theorem A]{KapovitchZiller04} classified such biquotients, and their results immediately imply that $M$ is diffeomorphic to $\s^{2n}$.

Finally, if $M$ admits a cohomogeneity one structure, it is rationally elliptic by Grove--Halperin \cite{GroveHalperin87}, hence $M$ is a rational sphere. Since $M$ has no two--torsion in its homology, it follows that $M$ is a mod $2$ homology sphere. The result now follows from Asoh's classification up to equivariant diffeomorphism of mod $2$ homology spheres that admit a cohomogeneity one action (see \cite[Main Theorem]{Asoh81}).



\begin{center}
\noindent
\begin{minipage}{\linewidth}
\small \noindent \textsc
{Manuel Amann} \\
\textsc{Fakult\"at f\"ur Mathematik}\\
\textsc{Institut f\"ur Algebra und Geometrie}\\
\textsc{Karlsruher Institut f\"ur Technologie}\\
\textsc{Kaiserstra\ss e 89--93}\\
\textsc{76133 Karlsruhe}\\
\textsc{Germany}\\
[1ex]
\textsf{manuel.amann@kit.edu}\\
\url{http://topology.math.kit.edu/21_54.php}
\end{minipage}
\end{center}

\vspace{10mm}

\begin{center}
\noindent
\begin{minipage}{\linewidth}
\small \noindent \textsc
{Lee Kennard} \\
\textsc{Department of Mathematics}\\
\textsc{South Hall}\\
\textsc{University of California}\\
\textsc{Santa Barbara, CA 93106-3080}\\
\textsc{USA}\\
[1ex]
\textsf{kennard@math.ucsb.edu}\\
\url{http://math.ucsb.edu/~kennard}
\end{minipage}
\end{center}

\newpage

\end{document}